\documentclass{amsart}

\newtheorem{thm}{Theorem}[subsection]

\newtheorem{lem}[thm]{Lemma}
\newtheorem{prop}[thm]{Proposition}
\theoremstyle{definition}
\newtheorem{defn}[thm]{Definition}
\newtheorem{example}[thm]{Example}
\newtheorem{problem}[thm]{Problem}

\theoremstyle{remark}

\numberwithin{equation}{subsection}
\usepackage{amsmath,amssymb,graphicx,bbm,amsthm}
\begin{document}

\title[satellite constructions and Geometric Classification of Brunnian links]
{satellite constructions and Geometric Classification of Brunnian links}

\author{Sheng Bai}
\address{College of Mathematics and Physics, Beijing University of Chemical Technology, Beijing, 100029, China}
\email{barries@163.com}

\author{Jiming Ma}
\address{School of Mathematical Sciences \\Fudan University\\
	Shanghai 200433, China} \email{majiming@fudan.edu.cn}


\thanks{Jiming Ma was partially supported by NSFC 11771088.}

\begin{abstract}
  We study satellite operations on Brunnian links. Firstly, we find two special satellite operations, both of which can construct infinitely many distinct Brunnian links from almost every Brunnian link. Secondly, we give a geometric classification theorem for Brunnian links, characterize the companionship graph defined by Budney in \cite{RB}, and develop a canonical geometric decomposition, which is simpler than JSJ-decomposition, for Brunnian links. The building blocks of Brunnian links then turn out to be Hopf $n$- links, hyperbolic Brunnian links, and hyperbolic Brunnian links in unlink-complements. Thirdly, we define an operation to reduce a Brunnian link in an unlink-complement into a new Brunnian link in $S^3$ and point out some phenomena concerning this operation.
\end{abstract}

\date{June. 25,  2020}
\maketitle

\section{Introduction}
The well-known geometric classification theorem for knots, proved by Thurston \cite{Th}, states that every knot type is in exactly one of the three classes, torus knots, hyperbolic knots, and satellite knots. Similarly, we have a generalization of Thurston's classification for the case of all links. Here by a \emph{knot}, we mean an  oriented circle in the 3-sphere upto ambient isotopies. A motivation of this paper is a refinement of this classification for the case of Brunnian links.

When provided with a link, we may naturally consider its sublinks. The nontrivial sublinks with least components are exactly knots and Brunnian links. In this sense, Brunnian links can be viewed as next in importance to knots for study of links. However, limited works has been done in Brunnian links. This is the first of a series of papers on Brunnian links. We are concerned in this paper with the satellites of Brunnian links and the objective is twofold. The first one is to construct infinitely many new Brunnian links using satellites. The other is to give deep geometric decomposition of Brunnian links by satellite patterns. We therefore present a general definition of Brunnian links.

\begin{defn}\label{def:brunn}
An $n$-component link $L$ in $S^3$ is \emph{Brunnian} if $n>1$, and

(ST):  every $(n-1)$-component sublink is trivial;

(NT):  $L$ is nontrivial.
\end{defn}

This paper has three main parts. Firstly, satellites of links have been defined and generalized in various ways (for instance \cite{RB,DFL,EN}). In this paper we give new names to two specific operations. The standard orientation of meridian and longitude used below will be described in the next section.

\begin{defn}\label{def:sum0}
Let $L$ and $L'$ be links in $S^3$, and $C \subset L$ and $C' \subset L'$ be oriented unknotted components. A homeomorphism $h: S^3 -int N(C') \longrightarrow N(C)$ maps the oriented meridian of $N(C')$  to the oriented longitude of $N(C)$ and maps the oriented longitude of $N(C')$  to the oriented meridian of $N(C)$, where $N(C)$ is the regular neighborhood of $C$. Then the link $(L-C) \cup h(L'-C')$ is the \emph{satellite-sum} of $L(C)$ and $L'(C')$, denoted $L_{C} \dagger_ {C'} L'$.
\end{defn}

\begin{defn}\label{def:tie0}
Let $L_0 \sqcup L^n$ be a link in $S^3$ where $L^n=\sqcup_{i=1}^n C^i$ is an oriented unlink, and $k_i$, $i=1,...,n$, be nontrivial knot types. Let $h: U_n =S^3 - int N(L^n) \longrightarrow S^3$ be an orientation-preserving embedding so that
\begin{align}
   S^3-h(U_n ) \cong \sqcup_{i=1}^n (S^3-N(k_i)), \nonumber
\end{align}
where $h(\partial N(C^i))$ corresponds to $\partial N(k_i)$, and the oriented meridian of $N(C^i)$ maps to the oriented null-homologous curve in $S^3-h(U_n)$ corresponding to the oriented longitude of $N(k_i)$. Then $L=h(L_0 )$ is a \emph{satellite-tie}, denoted $( L_0 , U_n \subset S^3 )\underrightarrow{k_1,...,k_n} L$.
\end{defn}

Satellite-sum and satellite-tie are in general different operations in that they correspond to unknotted and knotted tori in the complements of resulted links respectively. The \emph{splice} of links defined by Budney \cite{RB} includes both of them, and the notion of splice given by Eisenbud and Neuman \cite{EN} generalizes Budney's notion into homology spheres. Focusing on Brunnian links, we prove the following.

\begin{thm}\label{thm:sum0}
(1) Brunnian property is preserved by satellite-sum.

(2) If a torus $T$ in $S^3$ splits a Brunnian link $L$, then $T$ is incompressible, and $L$ is decomposed by $T$ as a satellite-sum of two Brunnian links.
\end{thm}

\begin{thm}\label{thm:tie0}
(1) Suppose $L$ is a satellite-tie of $( L_0 , U_n \subset S^3 )$, that is, $( L_0 , U_n \subset S^3 )\underrightarrow{k_1,...,k_n} L$. Then $L_0$ is Brunnian in $U_n$ if and only if $L$ is Brunnian in $S^3$.

(2) Every knotted essential torus in the complement of a Brunnian link bounds the whole link in the solid torus side.
\end{thm}

We emphasize that one can construct infinitely many distinct Brunnian links by satellite-sum from every Brunnian link, and by satellite-tie from almost every Brunnian link.

By Alexander's Theorem \cite{Ha}, every smooth torus bounds at least one solid torus in $S^3$.
\begin{defn}
Fixing an embedding of 3-manifold $M$ in $S^3$, a smooth torus in $M$ is \emph{knotted} if it bounds only one solid torus in $S^3$, and otherwise, \emph{unknotted}.
\end{defn}

\begin{defn}\label{def:prime0}
A Brunnian link is \emph{s-prime} if it is not Hopf link, and there is no essential unknotted torus in its complement space.
\end{defn}

\begin{defn}\label{defn:untied0}
A Brunnian link is \emph{untied} if there is no essential knotted torus in its complement space.
\end{defn}

Secondly, we classify Seifert-fibred Brunnian links and give the following geometric classification theorem of Brunnian links.

\begin{thm} \label{thm:classification0}
If a Brunnian link $L$ is s-prime, untied, and not a $(2, 2n)-$torus link, then $S^3 - L$ has a hyperbolic structure.
\end{thm}

We then further study the geometric decomposition of Brunnian links. In \cite{RB}, Budney describes a formalism for the JSJ decomposition of link complements by the construction of two graphs associated to a link, JSJ-graph and companionship graph. Briefly speaking, the companionship graph is a partially-directed tree with each edge labeled by a JSJ-decomposition torus, and each vertex labeled by a link whose complement is homeomorphic to the corresponding JSJ-piece. The graph with labels discarded is the JSJ-graph. We will characterize the two graphs of Brunnian links.

\begin{defn}
A partially-directed tree is a \emph{planting tree} if all of its vertices and directed edges form a disjoint union of rooted trees and the endpoints of each non-directed edge are roots.
\end{defn}

\begin{thm} \label{thm:jsj0}
The JSJ-graph of a Brunnian link is a planting tree.
\end{thm}

Moreover, we develop a notation of canonical decomposition for Brunnian links, simpler than the JSJ-decomposition.
If a link is a satellite-sum of certain links, we draw the tree structure of the satellite-sum expressed in terms of the links and removed components on a horizontal plane, and then replace the satellite-tie factors by their representation in Definition \ref{def:tie0}, the arrows drawn upward. We call such representation a \emph{tree-arrow structure}.
If, further, each factor with no arrow is either a hyperbolic Brunnian link or a $(2, 2n)-$torus link with $|n| \ge 2$, and each link in the satellite pattern is a hyperbolic Brunnian link in an unlink-complement, then we call this representation a \emph{Brunnian tree-arrow structure}.

By a Brunnian link in an unlink-complement,  we mean a nontrivial link $L$ in an unlink-complement such that all its  proper sublinks are trivial, see also Section 2. See Figure \ref{fig:treearrow} for an example of tree-arrow structure.

The main result in this part is

\begin{thm}\label{th:cano0}
Brunnian tree-arrow structures correspond bijectively to Brunnian links.
\end{thm}

Lastly, we are not satisfied with the occurrence of hyperbolic Brunnian links in unlink-complements as building blocks in the above theorem. So we define an operation, \emph{untie}, to reduce a Brunnian link in unlink-complement into a Brunnian link in $S^3$. We investigate some complicities of untying.

On the whole, given a Brunnian link in an unlink-complement, one can always untie it into untied, and then decompose it into s-prime factors in finite steps. Proposition \ref{prop:independent} in Section \ref{sec:gdbl} guarantees the resulted links are untied and s-prime. Combining with Theorem \ref{th:cano0} and \ref{thm:classification0}, in a word, we can say $(2, 2n)-$torus links and hyperbolic Brunnian links are the building blocks in the class of Brunnian links in that they generate all Brunnian links by satellite-sum and satellite-tie.

This paper is organized as follows. In Section 2, we present notations and some useful results. In Section \ref{sec:sum} and \ref{sec:tie}, we study satellite-sum and satellite-tie of Brunnian links, and prove Theorem \ref{thm:sum0} and \ref{thm:tie0} respectively. In Section \ref{sec:annuli}, we classify Seifert-fibred Brunnian links and prove Theorem \ref{thm:classification0} by discussion of essential annuli in link complements. We treat geometric decomposition of Brunnian links in Section \ref{sec:gdbl}. We characterize JSJ-graphs (Theorem \ref{thm:jsj0}) and companionship graphs, investigate the uniqueness of decomposition of satellite-sum and satellite-tie, and describe Brunnian tree-arrow structure (Theorem \ref{th:cano0}) by analysis on essential tori in this section. Finally, in Section \ref{sec:untie} we define and discuss untying.

\textbf{Acknowledgement}:
The authors would like to thank Zhiyun Cheng for helpful discussion in early time. The authors also appreciate the very helpful comments of an anonymous referee.

\section{Preliminaries}
In this paper, all objects and maps are smooth. All intersections are compact and transverse. We always consider links in $S^3$ if there is no extra claim. We use $L_i$ to denote a link,  $C_i$ to denote an unknot, $N(\cdot)$ to denote a closed regular neighborhood, $D_i$ to denote an embedded disk, and $A_i$ to denote an annulus.

A link is \emph{trivial} if it is the boundary of mutually disjoint disks. Given a link $L$ in $S^3$, the following statements are equivalent: (i) $L$ is trivial in $S^3$; (ii) $L$ is trivial in $R^3$, which is $S^3$ with a point removed; (iii) $L$ is trivial in a certain 3-ball.

Although the proof of the following proposition is trivial, the consequences of this result are of major importance for us.

\begin{prop}\label{prop:onedisk}
\cite{BW} Let $L$ be a link satisfying (ST) in Definition \ref{def:brunn} and $L_0$ be a proper sublink of $L$. Then $L$ is trivial if and only if $L_0$ bounds mutually disjoint disks not intersecting $L-L_0$.
\end{prop}

Recall that for a knot $K$ in $S^3$, there are two canonical isotopy classes of
curves in $\partial N(K)$, the \emph{meridian} and \emph{longitude} respectively. The meridian is the
essential curve in $\partial N(K)$  that bounds a disc in $N(K)$. The longitude is the essential curve in $N(K)$ that is null-homologous in $S^3-intN(K)$.  We orient the longitude of $N(K)$  similarly to $K$ and orient the meridian so that the longitude and the meridian has link number $+1$ as in \cite{RB}.

To investigate satellite-tie in Section \ref{sec:tie}, it is necessary to define \emph{Brunnian links in unlink-complements}. In a $3$-manifold, a  \emph{link} is a disjoint union of simple closed curves. A link is \emph{trivial} if it is the boundary of mutually disjoint disks. A link in a $3$-manifold is \emph{Brunnian} if it has more than one component, all of its proper sublinks are trivial, and itself is not trivial.

A 3-manifold $U_n$ is an \emph{(n-component) unlink-complement} if it is homeomorphic to $S^3 - int N(L^n)$, where $L^n$ is an $n$-component unlink for $n>0$. For example, if $L \subset S^3$ is a Brunnian link and $L'$ is a proper sublink with more than one component, then $L'$ is a Brunnian link in the unlink-complement $S^3 - int N(L-L')$.

In Section \ref{sec:annuli} we need Thurston's Hyperbolization Theorem. If a compact 3-manifold $M$ with toric boundary is irreducible, $\partial$-irreducible, anannular and atoroidal, then by Thurston's Geometrization for Haken 3-manifolds  \cite{Th},  $M$ admits a unique hyperbolic structure up to isometry.

We refer to \cite{Ha} for notions and details on 3-dimensional topology. Most arguments in Section \ref{sec:gdbl} rely on JSJ-Decomposition Theorem (mainly Torus Decomposition Theorem). The commonly used form of this theorem in our paper is: every oriented prime 3-manifold $M$ contains a canonical collection $\mathcal{T}_{JSJ}$ of disjoint essential tori, such that each component of $M-int N(\mathcal{T}_{JSJ})$ is either atoroidal or Seifert-fibred. Moreover, any maximal collection of disjoint essential tori with no parallel pair contains $\mathcal{T}_{JSJ}$ after isotopy.

We call each component of $M- int N(\mathcal{T}_{JSJ})$ a \emph{JSJ-piece}.

\section{Satellite-sum}\label{sec:sum}
\subsection{Definition and examples}
For simplicity of presentation, we always abbreviate satellite-sum by \emph{s-sum}. We name it as a sum since it is a satellite of both $L$ and $L'$.

Clearly, the s-sum is symmetric for $L(C)$ and $L'(C')$. S-sum is also associative for three links, with respect to the removed components. Accurately speaking, the s-sum of multiple links will form a (non-directed) tree structure. Let us illustrate it with an example.

\begin{example}\label{example:sum}
See Figure \ref{fig:sum}. There are 5 Brunnian links in (1). An s-sum of them, expressed in terms of the links and the removed components, forms a tree structure as shown in (2). The resulted link is (3).
\end{example}

\begin{figure}[htbp]
		  \centering
    \includegraphics[scale=0.18]{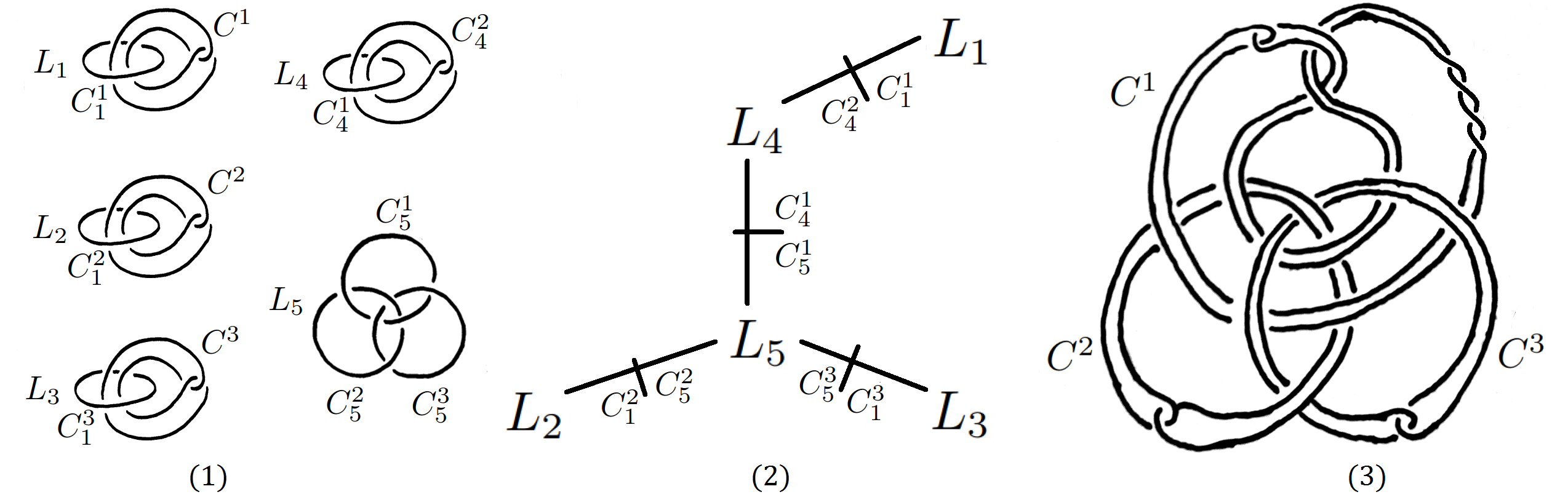}
    \caption{}
	\label{fig:sum}
\end{figure}

\begin{example}\label{example:milnor}
Milnor links (denoted $M^n$ if composed of $n$ components). It is observed from Figure \ref{fig:milnor} that, in general, $M^{n+m-2} = {M^n} _{C_r} \dagger _{C_l}  M^m$, with respect to the rightmost component $C_r$ of $M^n$ and the leftmost component $C_l$ of $M^m$. Notice that $M_3$ is Borromean rings. Performing the decomposition in succession, we obtain that $M_n$ is the s-sum of $n-2$ Borromean rings, and the tree structure is a path with $n-2$ vertices and $n-3$ edges.
\end{example}

\begin{figure}[htbp]
		  \centering
    \includegraphics[scale=0.3]{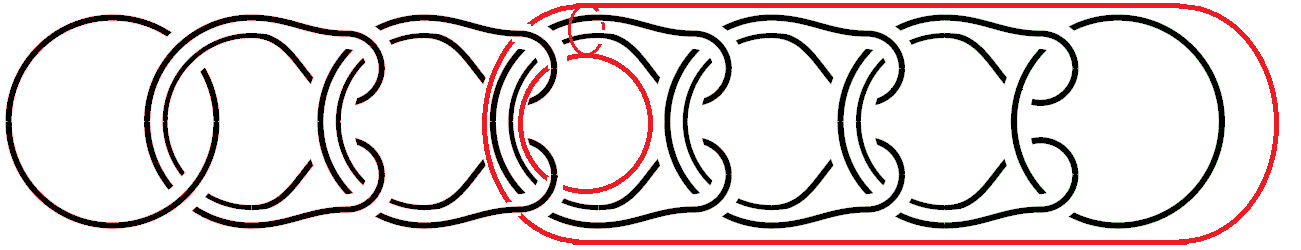}
    \caption{Milnor link: $M^7 = {M^4} _{C_r} \dagger _{C_l}   M^5$}
	\label{fig:milnor}
\end{figure}

Note that Hopf link is the identity in s-sum operations, as the s-sum of any link and Hopf link is itself.

\subsection{The special nature of s-sums of Brunnian links}
Recall from \cite{DR} that a compact set $K$ in a solid torus $V$ is \emph{geometrically essential} in $V$ if every meridian disk of $V$ intersects $K$. The following statements are equivalent: (i) $K$ is geometrically essential in $V$; (ii) there is no 3-ball in $V$ containing $K$; (iii) if $V$ is a standard solid torus in $S^3$, and $C$ is the core of $S^3-V$, then $C$ fails to bound a disk avoiding $K$.

The crucial result in this section is the following proposition.

\begin{prop}\label{prop:sum}
Let $L$ be an unlink, $C$ be an unknot not intersecting $L$, and $L'$ be a link in $N(C)$. If $L \cup L' \subset S^3$ satisfies (ST) in Definition \ref{def:brunn}, then the following statements are equivalent.

(i) $L \cup L'$ is Brunnian.

(ii) $L'$ and $L$ are geometrically essential in $N(C)$ and $S^3 - int N(C)$ respectively.

(iii) $L \cup C$ and $L' \cup C_0$ are Brunnian, where $C_0$ is the core of the solid torus $S^3 - int N(C)$.
\end{prop}

\begin{proof}
Before beginning to prove, we note that the positions of $L$ and $L'$ in the proposition are symmetric.

(i) $\Rightarrow$ (ii). We need only show that $L'$ is geometrically essential in $N(C)$. Assume otherwise.  Then there is a meridian disk $D$ of $N(C)$ not intersecting $L'$. In other words, $L'$ lies in a 3-ball subset of $N(C)$. Notice that $L \cup L'$ satisfies (ST). It follows that $L'$ bounds mutually disjoint disks in this 3-ball. According to Proposition \ref{prop:onedisk}, $L \cup L'$ is trivial.

(ii) $\Rightarrow$ (iii). We need only show that $L \cup C$ is Brunnian. First, we show that $L \cup C$ satisfies (ST). Let $C_i , i=1, ... ,n$ be the components of $L$. Since $L$ is trivial, it suffices to prove $\cup_{i=2}^n C_i \cup C$ is trivial without loss of generality.

Since $L \cup L'$ satisfies (ST), then $\sqcup_{i=2}^l C_i$ can bound disjoint disks $\sqcup_{i=2}^n D_i$ not intersecting $L'$. We may assume $\cup_{i=2}^n D_i \cap \partial N(C)$ is a disjoint union of circles. Delete all circles inessential on $\partial N(C)$ from innermost by surgery and denote the revised disks  still by $D_i$'s.

{\sc Case~1}. $\cup_{i=2}^n D_i \cap \partial N(C)= \emptyset$, namely, $\cup_{i=2}^n D_i \subset
S^3 - N(C)$. Then $\cup_{i=2}^n D_i$ is not geometrically essential in $S^3 - intN(C)$. Therefore $C$ also bounds a disk disjoint from $\cup_{i=2}^n D_i$ and thus $\cup_{i=2}^n C_i \cup C$ is trivial.

{\sc Case~2}.  $\cup_{i=2}^n D_i \cap \partial N(C) \neq \emptyset$. Then we choose an intersection circle $\tilde{c}$ innermost in some $D_i$, bounding $\tilde{D}$ on $D_i$. It follows that $\tilde{D}$ is a meridian disk of either $N(C)$ or $S^3 - intN(C)$. If $\tilde{D}$ is a meridian disk of $N(C)$, then $L'$ is not geometrically essential in $N(C)$, contradicting (ii). Hence, $\tilde{D}$ is a meridian disk of $S^3 - N(C)$. Then $\sqcup_{i=2}^n C_i$ is contained in a 3-ball, and thus bounds disjoint disks $\sqcup_{i=2}^n \bar{D_i}$ in it since $L \cup L'$ satisfies (ST). Clearly $C$ also bounds a disk disjoint from $\cup_{i=2}^n \bar{D_i}$ and thus $\cup_{i=2}^n C_i \cup C$ is trivial.

Now we show that $L \cup C$ is nontrivial. Assume $C$ bounds a disk disjoint from $L$, then there exists a meridian disk in $S^3 - intN(C)$ disjoint from $L$, contradicting that $L$ is geometrically essential in $S^3 - intN(C)$.

(iii) $\Rightarrow$ (i). Assume for contradiction that $L \cup L'$ is trivial. Then there is a disk $D_1$ bounded by a component of $L$, say $C_1$, disjoint from $(L-C_1) \cup L'$. We may assume $D_1 \cap \partial N(C)$ is disjoint union of circles. Delete all circles inessential on $\partial N(C)$ from innermost by surgery on $D_1$ and denote the revised disk still by $D_1$.

{\sc Case~1}. $D_1 \cap \partial N(C)= \emptyset$. Then $C_1$ bounds $D_1$ is disjoint from $(L-C_1) \cup C$, contradicting that $L \cup C$ is Brunnian.

{\sc Case~2}. $D_1 \cap \partial N(C) \neq \emptyset$. Then all the essential circles are parallel on $\partial N(L)$. Choose a circle $\tilde{c}$ innermost in $D_1$, bounding $\tilde{D} \subset D_1$. Then $\tilde{D}$ is either in $N(C)$ or $S^3 - intN(C)$, as a meridian disk. If $\tilde{D}$ is a meridian disk of $S^3 - intN(C)$, then $C$ can bound a disk disjoint from $L$, contradicting that $L \cup C$ is Brunnian. A similar argument shows that $\tilde{D}$ can not be a meridian disk of $N(C)$.
\end{proof}

An immediate consequence of this proposition is

\begin{thm}\label{thm:sum}
(1) Brunnian property is preserved by s-sum.

(2) If a torus $T$ in $S^3$ splits a Brunnian link $L$, then $T$ is incompressible, and $L$ is decomposed by $T$ as an s-sum of two Brunnian links.
\end{thm}

\begin{proof}
(1). Let $L$ and $L'$ be Brunian links, $C_i , i=1, ... ,n$ be the components of $L$, and $C'$ be a component of $L'$. Consider $L_{C_1} \dagger_ {C'} L'$. Applying the fact that (iii) $\Rightarrow$ (i) in Proposition \ref{prop:sum} to $L-C_1$ and $L'-C'$, we only must show that $(L-C_1) \cup (L'-C')$ satisfies (ST). Without loss of generality, it is sufficient to show that $\cup_{i=3}^n C_i \cup (L'-C')$ is trivial.

Since $L$ satisfies (ST), the sublink $\sqcup_{i=3}^n C_i$ is trivial in $S^3 - N(C_1)$. In particular, $\sqcup_{i=3}^n C_i$ bounds mutually disjoint disks in a 3-ball $B \subset S^3 - N(C_1)$. Therefore $L'-C'$ is contained in the 3-ball $S^3 - intB$ and thus trivial in it since $L'$ satisfies (ST).

(2). It follows from Theorem~\ref{thm:tie}(2) that $T$ is unknotted. Apparently the conclusion comes from the fact that (i) $\Rightarrow$ (iii) in Proposition \ref{prop:sum}.
\end{proof}

\section{Satellite-tie}\label{sec:tie}
\subsection{Definition and examples}\label{subsection:tiedefandexample}
For simplicity of presentation, we always abbreviate satellite-tie by \emph{s-tie}. We call $( L_0 , U_n \subset S^3 )$ in Definition \ref{def:tie0} an \emph{s-pattern} of $L$.

As an example, consider one of the simplest 2-component Brunnian links $L_0$ in a solid torus as shown in Figure \ref{fig:treborromean} (i). Notice that $U_1$ is a solid torus. The s-tie $( L_0 , U_1 \subset S^3 )\underrightarrow{k} L$, where $k$ is a trefoil, is a Brunnian link, as shown in Figure \ref{fig:treborromean} (ii). Varying the knot type of $k$ gives infinitely many distinct Brunnian links.

\begin{figure}[htbp]
		  \centering
    \includegraphics[scale=0.2]{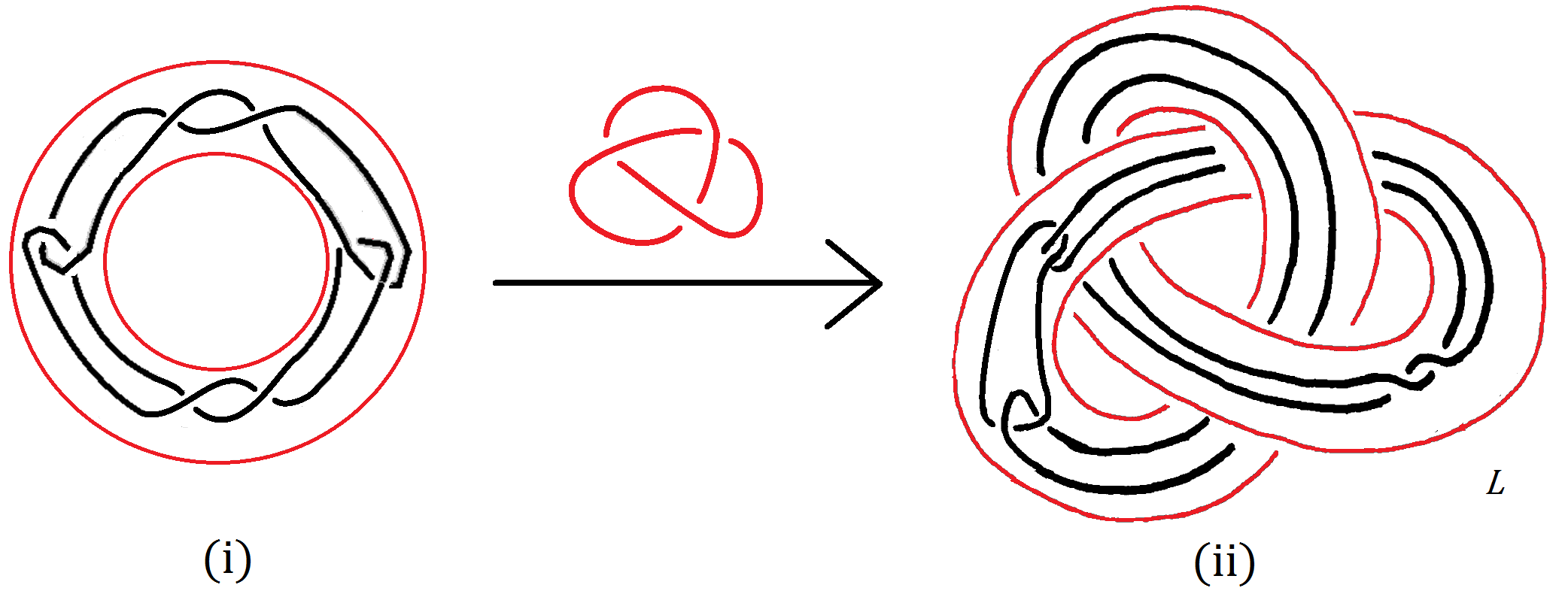}
    \caption{}
    \label{fig:treborromean}
\end{figure}

\subsection{Construction and the special nature of s-ties for Brunnian links}
The key idea of the proof of the main theorem in this section is the following Lemma \ref{lem:tie} (1).

\begin{lem}\label{lem:tie}
Let $K$ be a nontrivial knot, $L^r$ be an $r$-component unlink in $V=N(K)$ and $h$ be a homeomorphism from $U_{r+1}$ to $V - int N(L^r)$ where $U_{r+1}$ is an unlink-complement. Then for a link $L\subset U_{r+1}$,

(1) $L$ is  trivial in $U_{r+1}$ if and only if $h(L)$ is trivial in $S^3 - N(L^r)$;

(2) $L$ is Brunnian in $U_{r+1}$ if and only if $h(L)$ is Brunnian in $S^3 - N(L^r)$.
\end{lem}

\begin{proof}
(1). The ``only if '' implication is trivial, while the ``if '' implication follows when $h(L) \cup L^r$ is not geometrically essential in $V$, owing that $h(L) \cup L^r$ is contained in a 3-ball subset of $V$ and the triviality of $h(L)$ will be irrelevant with $V$.

Now suppose that $h(L) \cup L^r$ is geometrically essential in $V$ and $L=\sqcup_{i=1}^n C_i$ bounds mutually disjoint disks $\sqcup_{i=1}^n D_i$ in $S^3 - N(L^r)$. Then $\cup_{i=1}^n D_i$ cannot be contained in $V$ since otherwise $h(L) \cup L^r$ would become trivial in $V$ and then not geometrically essential in $V$. We may assume $\cup_{i=1}^n D_i \cap \partial V$ is a disjoint union of circles. Eliminate all circles inessential on $\partial V$ by surgery from innermost and denote the revised disks  still by $D_i$'s.

Choose an intersection circle $\tilde{c}$ essential on $\partial V$ and innermost on some $D_k$. Then $\tilde{c}$ bounds a disk $\tilde{D}$ on $D_k$. We claim this contradicts the assumptions. In fact, if $\tilde{D} \subset V$, then $h(L) \cup L^r$ is not geometrically essential in $V$. If $\tilde{D} \subset S^3 - intV$ then $K$ would be an unknot.

(2). Clause (1) implies that $L$ is not trivial in $U_{r+1}$ if and only if $h(L)$ is not trivial in $S^3 - N(L^r)$. Applying (1) to all the $(n-1)$-component sublinks of $L$, we obtain that all proper sublinks of $L$ are trivial in $U_{r+1}$ if and only if all proper sublinks of $h(L)$ are trivial in $S^3 - N(L^r)$.
\end{proof}

\begin{thm}\label{thm:tie}
(1) Suppose $( L_0 , U_n \subset S^3 )\underrightarrow{k_1,...,k_n} L$. Then $L_0$ is Brunnian in $U_n$ if and only if $L$ is Brunnian in $S^3$.

(2) Every knotted essential torus in the complement of a Brunnian link bounds the whole link in the solid torus side.
\end{thm}

\begin{proof}
(1). Let $h: U_n \longrightarrow S^3$ be the re-embedding in Definition \ref{def:tie0}, and $\sqcup_{r=1}^n  H^r = S^3 -inth(U_n)$ be the disjoint union of the knot complement spaces. Set $U_i= S^3 - (\sqcup_{r=1}^i  H^r)$ for $i= 0,1,...,n$. Then applying Lemma \ref{lem:tie}(2) for $i= 1,...,n$ in succession leads to the required conclusion.

(2). Given a Brunnian link $L$, suppose $T$ is a knotted torus in the complement space splitting $L$. Then the proper sublink in the solid torus $V$ bounded by $T$ is trivial in $S^3$ since $L$ satisfies (ST), and thus trivial in $V$ by Lemma \ref{lem:tie}(1). It follows from Proposition \ref{prop:onedisk} that $L$ is trivial. This is a contradiction.
\end{proof}

The ``only if'' implication of Theorem \ref{thm:tie}(1) gives a general method to construct infinitely many Brunnian links.

Theorem \ref{thm:tie}(2) implies that once there is a knotted essential torus in the complement of a Brunnian link, the link is an s-tie.

To construct infinitely many new Brunnian links, it is necessary to generalize the notion of geometrical essentiality into unlink-complements.

\begin{defn}
A compact set $K$ in an unlink-complement $U$ is \emph{geometrically essential} if $\partial U$ is incompressible in $U-K$.
\end{defn}

It is evident that $K$ is geometrically essential in the unlink-complement $U_n =S^3 - int N(\sqcup_{i=1}^n  C_i)$ if and only if for each $i$, $C_i$ fails to bound a disk avoiding $K\cup_{j \ne i} C_j$.

For example, if $L \subset S^3$ is a Brunnian link and $L'$ is a proper sublink, then $L-L'$ is geometrically essential in the unlink-complement $S^3 - int N(L')$.

We now see, given a geometrically essential Brunnian link in an unlink-complement $( L_0 , U_n \subset S^3 )$, one can always construct infinitely many distinct Brunnian links $( L_0 , U_n \subset S^3 )\underrightarrow{k_1,...,k_n} L$ by taking each $k_i$ to be every nontrivial knot type.

\subsection{Brunnian links in unlink-complements}\label{subsec:Bliuc}
We conclude this section with some observations on Brunnian links in unlink-complements. The diversity of Brunnian links in unlink-complements can be demonstrated by a typical example.

\begin{example}\label{example:unlinkcomplement}
The link $L$ in Figure \ref{fig:unlinkcomplement} (i) is Brunnian. Consider the red circles $C_1, C_2, C_3, C_4$ in Figure \ref{fig:unlinkcomplement} (ii). Then L becomes a Brunnian link in the unlink-complements $S^3 - N(C_i)$, $S^3 - N(C_i \sqcup C_j)$, $S^3 - N(C_i \sqcup C_j \sqcup C_k)$, $S^3 - N(C_1 \sqcup C_2 \sqcup C_3 \sqcup C_4)$ and so on. We can even ``twist'' or ``tie a knot'' of any type at the middle part of $C_4$. We therefore obtain infinitely many distinct Brunnian links in unlink-complements from $L$.
\end{example}

\begin{figure}[htbp]
		  \centering
    \includegraphics[scale=0.25]{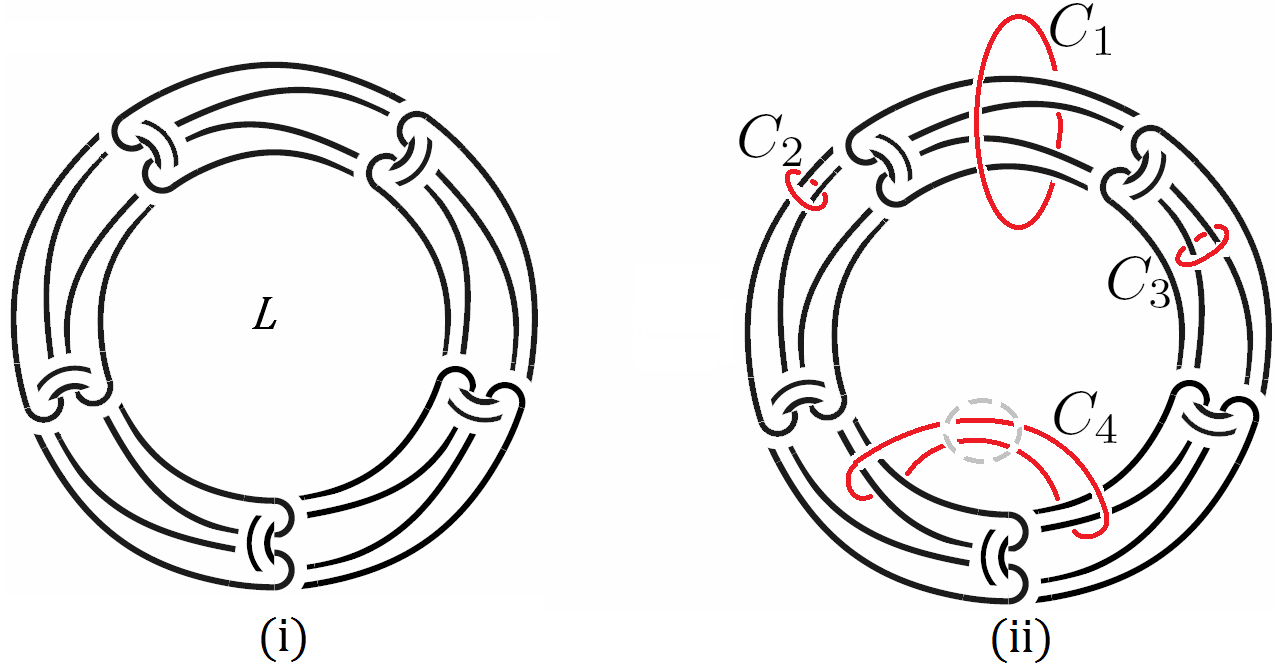}
    \caption{}
    \label{fig:unlinkcomplement}
\end{figure}

Before ending this section we show that one can always construct Brunnian links in infinitely many distinct unlink-complements, given a geometrically essential Brunnian link in an unlink-complement.

\begin{prop}\label{prop:unlinkcomplement}
Let $L$ be a geometrically essential Brunnian link in $U^n= S^3 - intN(L^n)$, where $L^n = \sqcup_{i=1}^n  C^i$ is an unlink, $L_1$, ..., $L_n$ be Brunnian links, and $C_i$ be a component of $L_i$ for each $i$. Let $\tilde{L}$ be an s-sum of $L^n $ and $L_1$,..., $L_n$, with respect to all pairs of $C^i$ and $C_i$. Then $L$ is a geometrically essential Brunnian link in the unlink-complement $\tilde{U} = S^3 -intN(\tilde{L})$.
\end{prop}

\begin{proof}
It is easy to check that $\tilde{L}$ is an unlink. Suppose $U_i$ is the complement space of $L_i$ for each $i$. Then $\tilde{U}$ is a gluing of $U^n$ and all of $U_i$'s via boundaries.

First we show that $L$ is Brunnian in $\tilde{U}$. Since $U^n \subset \tilde{U}$, $L$ satisfies (ST) in $\tilde{U}$. We prove $L$ is not trivial in $\tilde{U}$ by negation. Assume $L$ bounds mutually disjoint disks in $\tilde{U}$. Then the disks cannot be contained in $U^n$ since $L$ is not trivial in $U^n$. Thus some disk intersects $\partial N(L^n)$. We may isotope $N(L^n)$ to be thin enough in $S^3$ so that the disks intersect $N(L^n)$ only in their meridian disks. This contradicts the fact that $L_i$'s are nontrivial.

It remains to show that $L$ is geometrically essential in $\tilde{U}$. Assume otherwise. Without loss of generality let $C$ be a component of $L_1 -C_1$ bounding a disk in $\tilde{U}$. Then the disk cannot be isotoped into $U_1$ since $L_1$ is nontrivial. Thus the disk must contain a subdisk $D$ as a meridian disk of $S^3 - intN(C^1)$. If $ D \subset U^n$, then $L$ is not geometrically essential in $U^n$. Otherwise, we may isotope $N(L^n - C^1)$ to be thin enough so that $D$ intersects $N(L^n - C^1)$ only in their meridian disks, contradicting that other $L_i$'s are nontrivial.
\end{proof}

\section{Brunnian links with essential annuli in their complements}\label{sec:annuli}
In this section we will determine Brunnian links in $S^3$ with Seifert-fibered complements, and then prove the geometric classification theorem for Brunnian links.

We will call the $(2, 2n)-$torus link to be \emph{Hopf $n$-link}. Before proceeding further, we note that Hopf $n$-link has Seifert-fibered complement where $n$ is a nonzero integer. In addition, Hopf $n$-link with $|n| \ge 2$ is s-prime and untied because the complement space is atoroidal.

\begin{prop}\label{prop:annulus}
The Brunnian links with essential annulus joining different boundary components in their complements are Hopf $n$-links.
\end{prop}

\begin{proof}
Let $L=\cup_{i=1}^k C_i$ be a Brunnian link. Set $T_i = \partial N(C_i)$ for $i=1, 2$. Suppose $A \subset S^3 - intN(L)$ is an essential annulus joining essential curves $l_1$ on $T_1$ and $l_2$ on $T_2$.

First, we discuss the types of $l_1$ and $l_2$ on the two tori by considering the linking number $n=lk(C_1 ,C_2)$. Suppose $l_1$ is $(p, q)$-curve on $T_1$ and $l_2$ is $(p', q')$-curve on $T_2$. Notice that $l_1$ and $l_2$ are isotopic in the complement of $L$. They cannot both be $(0, 1)$-curves, i.e. meridians of $N(C_1)$ and $N(C_2)$, since these meridians are in different homology classes of $S^3 - intN(L)$. We have
\begin{align*}
lk(l_1, C_1)=q,  lk(l_2, C_1)=np';\\
lk(l_1, C_2)=np, lk(l_2, C_2)=q' .
\end{align*}

This implies that
\begin{align*}
q =np',    q'=np.
\end{align*}
We discuss the link type of $L$ by cases.

{\sc Case~1}. $|p|=1$. Then since $l_1$ is an unknot, so does $l_2$. Thus we have $|p'|=1$ or $|q'|=1$.

{\sc Subcase~1.1}. $|p'|=1$. Then $l_i$ is a $\pm (1, n)$-curve on $T_i$ for $i=1, 2$, and thus $C_1$ is parallel to $C_2$. If $n=0$, then $L - C_1$ is trivial would implies $L$ is trivial. So $n\neq 0$ and $L=C_1 \cup C_2$ is the Hopf $n$-link.

{\sc Subcase~1.2}. $|q'|=1$. Then $n=\pm 1$, $l_1$ is a $\pm (1, q)$-curve and $l_2$ is a $\pm (q, 1)$-curve.

{\sc Case~2}. $|q|=1$. Then an argument similar to the one used in Case 1 shows that $n=\pm 1$, $l_1$ is a $\pm (p, 1)$-curve, and $l_2$ is a $\pm (1, p)$-curve.

{\sc Case~3}. $pq=0$. Then either $l_1$ is $(0, 1)$-curve and $l_2$ is $(1, 0)$-curve or verse visa. In either subcase, $L$ is the Hopf link.

{\sc Case~4}. $|p|, |q| >1$. Then $l_1$ is a torus knot. Since $l_2$ is isotopic to $l_1$, by classification of torus knots, $l_2$ must be either $\pm(p, q)$-curve or $\pm(q, p)$-curve on $T_2$. The first subcase is impossible. In the latter subcase, $n=\pm 1$.

In summary, we have proved that $L$ is Hopf $n$-link unless $n=\pm 1$, $l_1$ is a $(p, q)$-curve, and $l_2$ is a $\pm (q, p)$-curve provided that $pq \ne 0$.

Now, without loss of generality, suppose $L=C_1 \cup C_2$, $n= 1$, $l_1$ is $(p, q)$-curve on $T_1$, and $l_2$ is $(q, p)$-curve on $T_2$ provided that $pq \ne 0$. We will prove $L$ is the Hopf link.

Set $V_1 =S^3 - intN(C_1)$. Then $A \cup N(C_2) \subset V_1$. Choose a meridian disk $D$ of $V_1$ such that $\partial D \cap l_1$ consists of $q$ points. Then $D \cap T_2$ is a disjoint union of circles. First, we eliminate all circles inessential on $T_2$ from innermost by isotopy of $T$. Then the remained circles, denoted by $\sqcup_{i=1}^r \bar{C_i}$, have to be meridians of $N(C_2)$ since they are parallel on $T_2$ and an innermost one on $D$ bounds a meridian disk of $N(C_2)$. Next, consider $\partial D \cap A$, a disjoint union of circles and proper arcs.

We claim that each proper arc cannot have both endpoints on $\partial D$. Assume otherwise. Then such an arc $\alpha$, outermost on $A$, cuts off a disk on $A$. We can isotope $D$ to push forward this disk to eliminate $\alpha$, contradicting that the number of $\partial D \cap \partial A$ has achieved minimum.

Therefore, the circles are inessential on $A$ and can be eliminated from innermost by isotopy of $A$. Using the same argument as in the previous paragraph, each proper arc with both endpoints on some $\bar{C_i}$ can also be eliminated.

It remains to eliminate each proper arc connecting two different $\bar{C_i}$'s. Start from such an arc $\beta$ outermost on $A$. Then it cuts off a disk $D_\beta$ on $A$. An isotopy of a neighborhood on $D$ of $\beta$ to push $\beta$ across $D_\beta$ eliminates $\beta$. We need to eliminate the occurred circles inessential on $T_2$. Now $D \cap A$ is a disjoint union of proper arcs each of which connects $\partial D$ and some $\bar{C_i}$'s.

Notice that $L_2$ is a $(q, p)$-curve. On each $\bar{C_i}$ the number of endpoints is no less than $q$. So there is only one $\bar{C_i}$ and thus $C_2$ is the core of $V_1$. It follows that $L$ is the Hopf link.
\end{proof}

\begin{prop}\label{prop:seifert}
The only Brunnian links with Seifert-fibred complements are Hopf $n$-links.
\end{prop}

\begin{proof}
Let $L$ be a Brunnian link such that $S^3 - intN(L)$ is a Seifert manifold with singular fibers $f_i, i=1,...,t$. We only need to find an essential annulus in $S^3 - intN(L)$. Let $p: S^3 - N(L \cup_{i=1}^t f_i) \rightarrow B$ be the circle bundle. Let $T_i = \partial N(C_i)$ consist of fibers for $i=1, 2$ where $C_i$'s are components of $L$. Choose a simple arc $\alpha$ in $B$ joining $p(T_1)$ and $p(T_2)$. Then $A=p^{-1} (\alpha)$ is an annulus in $S^3 - intN(L)$ joining essential curves on $T_1$ and on $T_2$. Clearly $A$ is essential for otherwise $C_1$ would be a trivial component.
\end{proof}

We point out that Burde and Murasugi \cite{BM} (also Budney \cite[Section 3]{RB}) classified links in $S^3$ with Seifert-fibered complements. We can also use their result to obtain Proposition \ref{prop:seifert}.

We are now in a position to prove the geometric classification result for Brunnian links.

\begin{thm} \label{thm:classification}
If a Brunnian link $L$ is s-prime, untied, and not a Hopf $n$-link, then $S^3 - L$ has a hyperbolic structure.
\end{thm}

\begin{proof}
The complement space of a Brunnian link is irreducible and $\partial$-irreducible by Proposition \ref{prop:onedisk}. Then the theorem follows immediately from Proposition \ref{prop:annulus} and Thurston's Hyperbolization Theorem.
\end{proof}

The next proposition will be used in the next section.

\begin{prop}\label{prop:annulus2}
Let $L$ be a Brunnian link in an unlink-complement $U$. Then in the complement space $U-intN(L)$, there is  no essential annulus joining $\partial U$ and $\partial N(L)$, and thus $U-intN(L)$ is not Seifert-fibered.
\end{prop}

\begin{proof}
This proposition can be proved in the same way as shown in Proposition \ref{prop:annulus} and Proposition \ref{prop:seifert}.
\end{proof}

\section{Geometric decomposition of Brunnian links}\label{sec:gdbl}
In this section, we will characterize JSJ-graphs and develop a notation of canonical geometric decomposition for Brunnian links.

\subsection{JSJ-graphs of Brunnian links}\label{subset:JSJ}
We will use the terminologies of \cite{RB} in this subsection.

\begin{defn}
A partially-directed tree is a \emph{planting tree} if all of its vertices and directed edges form a disjoint union of rooted trees and the endpoints of each non-directed edge are roots.
\end{defn}

We sketch a 3-dimensional planting tree in Figure \ref{fig:planting}.

\begin{figure}[htbp]
		  \centering
    \includegraphics[scale=0.3]{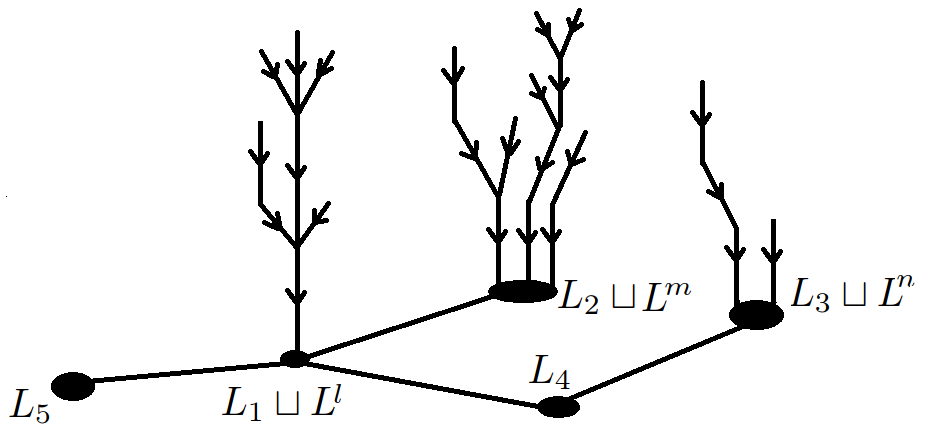}
    \caption{A 3-dimensional graph of planting tree.}
    \label{fig:planting}
\end{figure}

\begin{thm} \label{thm:jsj}
The JSJ-graph of a Brunnian link is a planting tree.
\end{thm}

\begin{proof}
This follows easily from Theorem \ref{thm:tie} (2), the fact that every unknotted essential torus splits the link, and that the JSJ-graph of a knot is a rooted tree \cite{RB}.
\end{proof}

Specifically, in the companionship graph of a Brunnian link \cite{RB}:

(1) each leaf incident to a directed edge is labeled by a knot; each vertex neither a leaf nor a root is labeled by a KGL \cite{RB};

(2) the directed edges terminating to a vertex correspond together to a trivial sublink of the link labelling this vertex.

Moreover, as will be shown at the end of this section, we have

(3) Each root is labeled by a Hopf $n$-link, a hyperbolic Brunnian link, or a hyperbolic Brunnian link in unlink-complement. In accurate, if the pattern here is $( L_1 , U_l \subset S^3 )$ and $U_l = S^3 - int N(L^l)$, then the root is labeled by $L_1 \cup L^l$.

It is not difficult to identify that companionship graphs satisfying the above properties correspond bijectively to Brunnian links.

\subsection{Decomposition of s-sum and s-tie and their relationship}\label{subset:fine}
We now turn to investigations on decomposition of s-sum and s-tie of Brunnian links.

First consider the s-sum decomposition. It is not surprising that

\begin{prop}\label{prop:sumtree}
Every Brunnian link, except Hopf link, can be represented uniquely as an s-sum of s-prime Brunnian links, with respect to the pairs of removed components.
\end{prop}

In other words, the s-sum tree structures, expressed in terms of s-prime Brunnian links, correspond bijectively to Brunnian links, except Hopf link. We point out that the conclusion does not hold for general links in view of Seifert-fibred links such as key-chain links \cite{RB}.

We would rather prove it directly than employ the Torus Decomposition Theorem. A direct proof will be given in Appendix and a proof by Torus Decomposition Theorem will be outlined in the next subsection. We will deal with Proposition \ref{prop:fine} and Lemma \ref{lem:isotopy} in the same manner.

Next, to explore the uniqueness of s-tie decomposition, it is necessary to note a special class of knotted essential tori in link complements.

\begin{defn}\label{def:fine}
For a link, a knotted essential torus $T_0$ in the link complement is a \emph{fine companion} if there is no knotted essential torus $T$ so that

(i)  $T \subset V_0$, where $V_0$ is the solid torus bounded by $T_0$ in $S^3$.

(ii) $T$ bounds a solid torus in $V_0$ and $T$ is not parallel to $T_0$.
\end{defn}

Informally, a fine companion is an innermost knotted essential torus in a link complement. By \cite[Proposition 2.1]{RB}, a disjoint union of fine companions with no parallel pair always bounds an unlink-complement.

\begin{prop}\label{prop:fine}
For a Brunnian link, the maximal collection of fine companions with no isotopic pair is finite, and the tori are disjoint after isotopy.
\end{prop}

We point out that, for a Brunnian link, the maximal collection of knotted essential tori with no isotopic pair may be neither finite nor mutually disjoint after isotopy, in view of an occurrence of Seifert-fibred piece in the JSJ-decomposition. Furthermore, the conclusion does not hold for knots in view of a connected sum of 3 prime knots as counterexample.

Another investigation is the relationship between s-sum and s-tie. In the rest of this paper, for a Brunnian link $L$, we let $\mathcal{T}_{sum} (L)$ denote the maximal collection of disjoint tori in the s-sum decomposition, and $\mathcal{T}_{fine} (L)$ denote the maximal collection of disjoint fine companions with no parallel pair.

\begin{lem}\label{lem:isotopy}
Let $L$ be a Brunnian link, $\mathcal{T}_{knotted}$ and $\mathcal{T}_{unknotted}$ be collections of disjoint knotted and unknotted essential tori in the link complement respectively. Then the tori in $\mathcal{T}_{knotted}$ are disjoint with the tori of $\mathcal{T}_{unknotted}$ after isotopy.
\end{lem}

\begin{prop}\label{prop:independent}
Let $L$ and $L'$ be Brunnian links, $C \subset L$ and $C' \subset L'$ be components, and $h$ be the map in Definition \ref{def:sum0}. Then $\mathcal{T}_{fine} (L_{C} \dagger_ {C'} L')= \mathcal{T}_{fine} (L) \sqcup h(\mathcal{T}_{fine} (L'))$.
\end{prop}

\begin{proof}
It suffices to show that (i) each knotted essential torus in the complement of $L$ is a knotted essential torus in the complement of $L_{C} \dagger_ {C'} L'$; (ii) each knotted essential torus in the complement of $L_{C} \dagger_ {C'} L'$ is inherited from the complement of either $L$ or $L'$.

Set $V=N(C)$ containing $h(L'-C')$ and $V'=S^3 - intN(C)$ containing $L-C$. For (i), we need only show that such a torus in $V'$ is still essential in the complement of $L_{C} \dagger_ {C'} L'$. Notice that $h(L'-C')$ is geometrically essential in $V$. A routine argument gives (i). For (ii), we may assume such a torus in the complement of $L_{C} \dagger_ {C'} L'$ is in $V'$ by the previous lemma. Clearly it is essential in the complement of $L$.
\end{proof}

\subsection{Essential tori in the complement of a Brunnian link}\label{subsec:tori}
We will give a detailed analysis of essential tori in the complements of Brunnian links. Given a Brunnian link $L$,  rather than consider a single torus, we consider a maximal collection of disjoint essential tori with no parallel pair, say $\mathcal{T}_{maximal}$.

Let $\mathcal{T}_{knotted}$ and $\mathcal{T}_{unknotted}$ be the subsets of knotted and unknotted tori in $\mathcal{T}_{JSJ}$ respectively, and $\mathcal{T'}_{knotted}$ and $\mathcal{T'}_{unknotted}$ be the subsets of knotted and unknotted tori in $\mathcal{T}_{maximal}$ respectively. Set $\mathcal{T}_{canonical} = \mathcal{T}_{sum} (L) \sqcup \mathcal{T}_{fine} (L)$.

We will show the following three facts. It is desirable to consider the JSJ-graph and use the Torus Decomposition Theorem. The main point of our argument is that each JSJ-piece corresponding to a root is atoroidal. It can by verified by Theorem \ref{thm:classification} and Proposition \ref{prop:annulus2}.

(1) $\mathcal{T}_{sum} = \mathcal{T}_{unknotted}=\mathcal{T'}_{unknotted}$. In fact, since each JSJ-piece corresponding to a root is atoroidal, it is easy to show that $\mathcal{T}_{sum} \subset \mathcal{T}_{unknotted}$. By the Torus Decomposition Theorem, we directly have $\mathcal{T}_{unknotted} \subset \mathcal{T'}_{unknotted}$, while Theorem \ref{thm:sum}(2) implies $\mathcal{T'}_{unknotted} \subset \mathcal{T}_{sum}$. This also outlines a proof of Proposition \ref{prop:sum}, and then Lemma \ref{lem:isotopy} follows from the Torus Decomposition Theorem.

(2) $\mathcal{T}_{fine}$ labels exactly the directed edges adjacent to roots. In fact, we may choose $\mathcal{T}_{maximal}$ so that it contains $\mathcal{T}_{fine}$. Then by the Torus Decomposition Theorem, $\mathcal{T}_{knotted} \subset \mathcal{T'}_{knotted}$. By definition of $\mathcal{T}_{fine}$, the unlink-complement bounded by it contains no knotted essential torus. On the other hand, the JSJ-piece corresponding to a root also contains no fine companion since it is atoroital. Thus the claim and Proposition \ref{prop:fine} follow.

(3) $\mathcal{T}_{maximal} - \mathcal{T}_{JSJ}$ can only appear in JSJ-pieces homeomorphic to the complements of key-chain links. This needs a careful examination on all of possible Seifert-fibred JSJ-pieces with the help of \cite{BM} (or \cite[Section 3]{RB}).

\begin{figure}[htbp]
		  \centering
    \includegraphics[scale=0.2]{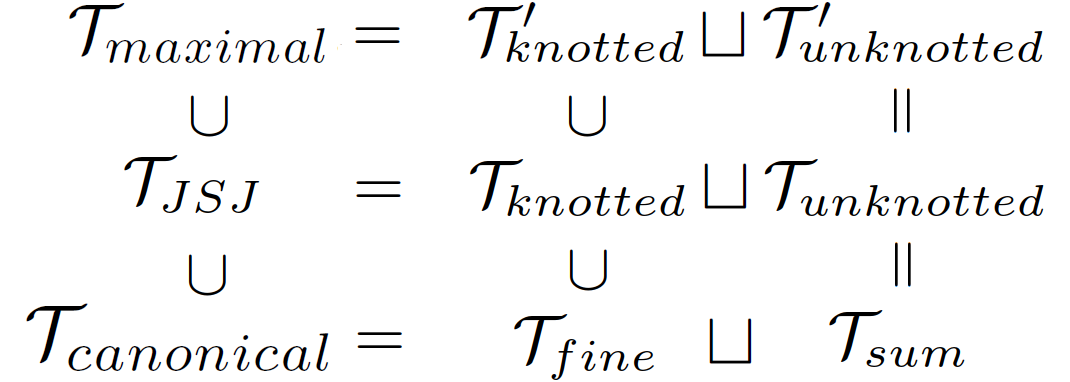}
    \caption{The containment relationship of the nine collections of tori.}
    \label{fig:jsj}
\end{figure}

\subsection{Tree-arrow structure}
Now we develop a notation of canonical geometric decomposition, simpler than the JSJ-decomposition, for Brunnian links.

\begin{defn}\label{def:cano}
If a link is an s-sum of certain links, we draw the tree structure of the s-sum expressed in terms of the links and removed components on a horizontal plane, and then replace the s-tie factors by their representation in Definition \ref{def:tie0}, the arrows drawn upward. We call such representation a \emph{tree-arrow structure}.

Moreover, if each factor with no arrow is either a hyperbolic Brunnian link or a Hopf $n$-link with $|n| \ge 2$, and each link in the s-pattern is a hyperbolic Brunnian link in unlink-complement, then we call this representation a \emph{Brunnian tree-arrow structure}.
\end{defn}

We sketch a 3-dimensional tree-arrow structure in Figure \ref{fig:treearrow}.

\begin{figure}[htbp]
		  \centering
    \includegraphics[scale=0.3]{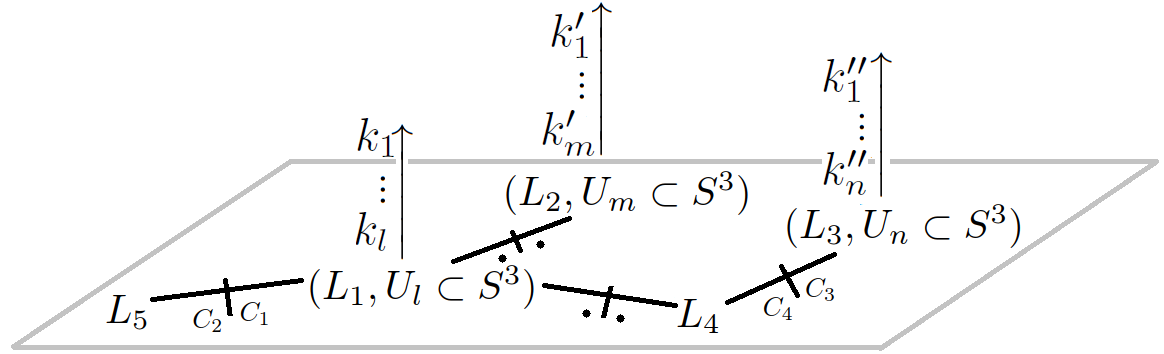}
    \caption{A 3-dimensional graph of tree-arrow structure.}
    \label{fig:treearrow}
\end{figure}

We present two equivalent descriptions of Brunnian tree-arrow structure.

Given a Brunnian link $L$, other than Hopf link, $\mathcal{T}_{sum} (L)$ decomposes it into s-prime factors $L_1, ... ,L_n$. Take the s-pattern of $L_i$ corresponding to $\mathcal{T}_{fine} (L_i)$ as long as $\mathcal{T}_{fine} (L_i) \ne \emptyset$. Then we obtain a Brunnian tree-arrow structure due to Theorem \ref{thm:classification}, Proposition \ref{prop:annulus2} and Thurston's Hyperbolization Theorem. Moreover, $\mathcal{T}_{fine} (L) = \sqcup \mathcal{T}_{fine} (L_i)$.

Combining (1) and (2) in the previous subsection, we see Brunnian tree-arrow structure is essentially the companionship graph with each maximal rooted tree replaced by its representative knots.

It is sure that every Brunnian tree-arrow structure gives a Brunnian link by Theorem \ref{thm:sum} (1) and Theorem \ref{thm:tie} (1). From the discussion of the previous two paragraphs, we may assert:

\begin{thm}\label{th:cano}
Brunnian tree-arrow structures correspond bijectively to Brunnian links.
\end{thm}

\section{Untying Brunnian links}\label{sec:untie}
Hyperbolic Brunnian links in unlink-complements occur as building blocks in Brunnian tree-arrow structure. However, Example \ref{example:unlinkcomplement} indicates that Brunnian links in unlink-complements may be too diverse in general. This motivates the following operation to transform a Brunnian link in unlink-complement into $S^3$.

\begin{defn}\label{def:untie}
Let $L$ be a   Brunnian link in $S^3$, and $T$ be a torus in the complement $S^3-L$, which bounds  a solid torus $V$ such that $L \subset V$.
We perform \emph{untying $T$ for $L$} if we re-embed $V$ into $S^3$ in the unknotted way by a map $h$ such that the longitude of $V$ is mapped to the meridian of $S^3-int h(V)$, and

type-1: take $h(L) \sqcup C$, if $h(L)$ is trivial in $S^3$, where $C$ is the core of $S^3-h(V)$;

type-0: take $h(L)$, if $h(L)$ is nontrivial in $S^3$.

\end{defn}

It is straightforward to show that the resulted link is Brunnian. Since untying inessential torus gives the original link, we only focus on essential torus.

Set $\mathcal{T}_{fine} (L)= \{T_1 ,T_2 , ... ,T_l \}$ for a Brunnian link $L$. The basic question is how the order of untying tori affects the resulted link. Three questions arise:

(1) Is the type of untying invariant under the change of order?

(2) After untying $T_1$, does $T_2$ must remain essential?

(3) After untying all fine companions of $L$, is the resulted link untied?

They all have negative answers as will be shown in Examples \ref{example:untie1}, \ref{example:untie2}, and \ref{example:untie3} respectively.

\begin{example}\label{example:untie1}
Figure \ref{fig:order}. If we untie $T_1$ before $T_2$, then $T_1$ is untied in type-0 and $T_2$ is untied in type-1. On the other hand, if we untie $T_2$ first, then we should untie $T_2$ in type-0.
\end{example}

\begin{figure}[htbp]
		  \centering
    \includegraphics[scale=0.25]{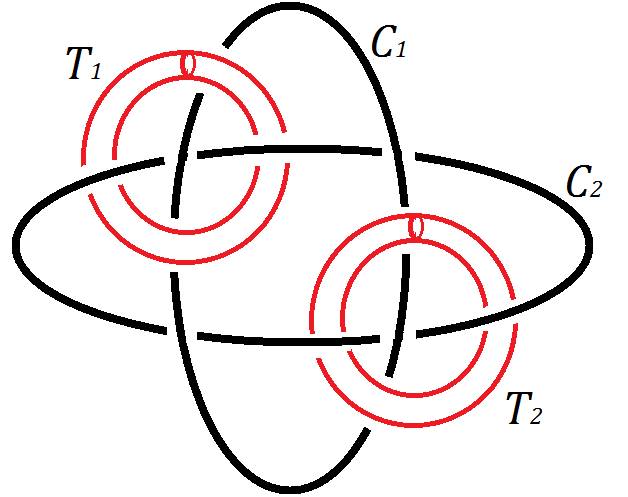}
    \caption{$(C_1 \sqcup C_2 , U_2 \subset S^3 )$ is an s-pattern.}
    \label{fig:order}
\end{figure}

\begin{example}\label{example:untie2}
Figure \ref{fig:t1t2}. Untie $T_1$, then $T_2$ becomes inessential.
\end{example}

\begin{figure}[htbp]
		  \centering
    \includegraphics[scale=0.25]{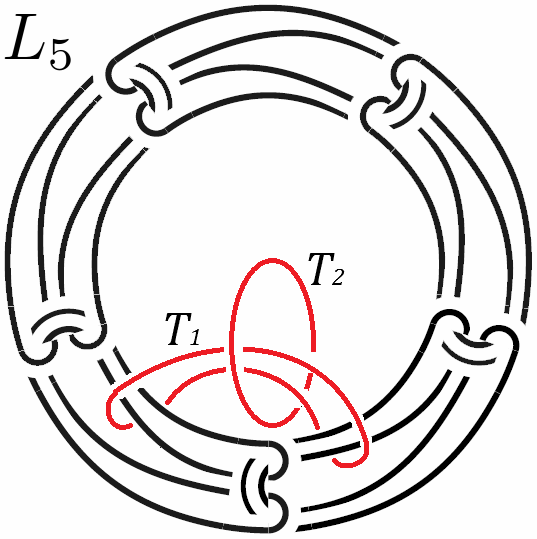}
    \caption{$(L_5, U_2 \subset S^3 )$ is an s-pattern. $T_1$, $T_2$ are the boundary of the regular neighborhood of the two red circles.}
    \label{fig:t1t2}
\end{figure}

\begin{example}\label{example:untie3}
Figure \ref{fig:tienew}. After untying all fine companions, a new essential knotted torus emerges.
\end{example}

\begin{figure}[htbp]
		  \centering
    \includegraphics[scale=0.25]{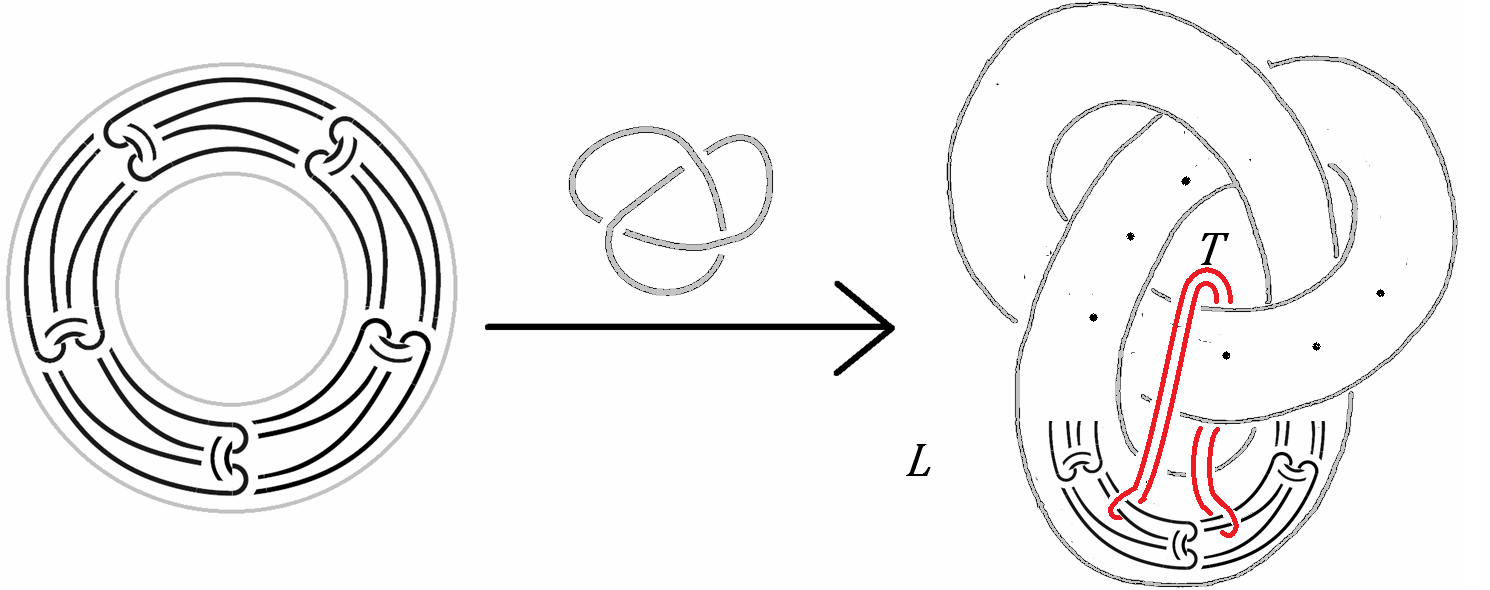}
    \caption{The right figure is an s-pattern$ (L, U_1 \subset S^3 )$. The boundary of regular neighborhood of the red circle is $T$. After untying $T$, i.e. the red circle is filled, the link is still an s-tie.}
    \label{fig:tienew}
\end{figure}

Fortunately, we have partially affirmative answers to questions (1) and (3).

\begin{prop}\label{prop:finite}
(1) Let $T_1 ,T_2 , ... ,T_l $ be fine companions for a Brunnian link $L$. Suppose $L_0$ is the resulted link after untying $T_1 ,T_2 , ... ,T_{l-1}$ in order. If we should untie $T_l$ for the original $L$ in type-1, then we should now untie $T_l$ for $L_0$ in type-1.

(2) For a Brunnian link, if in each step we untie an essential knotted torus for the obtained link, then the procedure stops in finite steps.
\end{prop}

\begin{proof}
(1). Let $V_l$ be the solid torus bounded by $T_l$ and the unlink-complement bounded by $T_1 \cup T_2 \cup ... \cup T_{l-1}$ be $U_0$. It is equivalent to prove that if the re-embedding $h$ of $V_l$ makes $h(L)$ trivial, then the re-embedding by the same choice of longitude on $T_l$ will make $L_0$ trivial. Untie $T_l$ first. By Theorem \ref{thm:tie} (1), if $h(L)$ is trivial in $S^3$, then it is trivial in $U_0$.

(2). In each step, take all the link complement pieces in the Brunnian tree-arrow structure of the obtained link. There are no essential knotted tori in Seifert pieces by Proposition \ref{prop:annulus} and Proposition \ref{prop:annulus2}. For each hyperbolic piece, untying either keeps or decreases the hyperbolic volume. The volumes of hyperbolic 3-manifolds have a uniform  lower bound, in fact it is of order type $\omega^\omega$, see \cite{Th}. So the procedure stops in finite steps.
\end{proof}

We finally remark that the main theorem of \cite{MS} implies that in Definition \ref{def:untie}, the way of re-embedding making $h(L)$ trivial is unique. Using this theorem inductively with the help of $s$-tie construction, we confirm that we may always drawn Brunnian links in unlink-complements like Figure \ref{fig:unlinkcomplement}. We will omit the details here.

\begin{prop}
For every Brunnian link $L$ in unlink-complement $U_n$, there exists an embedding $h: U_n \rightarrow S^3$ (actually infinitely many) so that $S^3 - inth(U_n)$ is a regular neighborhood of unlink and $h(L)$ is Brunnian in $S^3$.
\end{prop}

In conclusion, study of Brunnian links in unlink-complements can be reduced to that in $S^3$, as what we have hoped. However, the following problem is basic but seems challenging.

\begin{problem}
Given a Brunnian link $L$, find or classify all the unlinks not intersecting $L$ such that $L$ is Brunnian in their complements.
\end{problem}

From Theorem \ref{thm:classification},  we ask the question of the probability of a  Brunnian link to be a hyperbolic link.

Note that there are different models for random links. For random links via random braids or via random bridge decompositions, a generic link is a hyperbolic link \cite{IchiharaMa, Ito, Ma2014}.  But surprisely, via the more combinatorial model, that is, from the viewpoint of crossing number,  hyperbolic links are not generic \cite{Malyutin2018, Malyutin2019}. Note that   Brunnian links are determined by their complement  when there are  at least three components \cite{MS}. Inspiried by  \cite{Malyutin2018}, we propose the following problem.

\begin{problem}Let $B(n)$ be the number of Brunnian links with  $n$ or fewer crossings, and $B_{h}(n)$ be the number of hyperbolic Brunnian links with  $n$ or fewer crossings. Then
$$\limsup_{n \rightarrow \infty} \frac{B_{h}(n)}{B(n)}=a, $$
$$\liminf_{n \rightarrow \infty} \frac{B_{h}(n)}{B(n)}=b.$$
We have $0 \leq b \leq a \leq 1$  trivially, we believe that   $a=b$. Is $b=0$ or $a <1$?

\end{problem}

\bibliographystyle{amsalpha}

\begin{thebibliography}{10}

\bibitem{BFJRVZ}
N. A. Baas, D. V. Fedorov, A. S. Jensen, K. Riisager, A. G. Volosniev, and N. T. Zinner.
\newblock {\em  Higher-order Brunnian structures and possible physical realizations.}
\newblock arXiv:1205.0746v2 [quant-ph] 4 May 2012.

\bibitem{BCS}
N. A. Baas,  N. C. Seeman,  Andrew Stacey.
\newblock {\em  Synthesising topological links.}
\newblock J. Math. Chem. (2015) 53:183--199.

\bibitem{BS}
N. A. Baas and A. Stacey.
\newblock {\em Investigations of higher order links.}
\newblock arXiv:1602.06450v1 [math.AT].

\bibitem{BW}
S. Bai and W. Wang.
\newblock {\em  New methods to detect Brunnian property of links.}
\newblock  in preparation.

\bibitem{Br}
H. Brunn.
\newblock {\em Uber Verkettung, Sitzungsberichte der Bayerische Akad. Wiss.}
\newblock Math-Phys. Klasse, 22 (1892), 77--99.

\bibitem{RB}
R. Budney.
\newblock {\em  JSJ-decompositions of knot and link complements in $S^3$}.
\newblock Enseign. Math. 3(2005), 319--359.

\bibitem{BM}
G. Burde and K. Murasugi.
\newblock {\em Links and Seifert fiber spaces.}
\newblock Duke Math. J., 37:89--93, 1970.

\bibitem{DFL}
A. Degtyarev, V. Florens, and A. G. Lecuona.
\newblock {\em  The signature of a splice.}
\newblock International Mathematics Research Notices, Vol. 2017, No. 8, pp. 2249--2283.

\bibitem{EN}
D. Eisenbud and W. Neumann.
\newblock {\em  Three-Dimensional Link Theory and Invariants of Plane Curve Singularities.}
\newblock Annals of Mathematics Studies 110. Princeton, NJ: Princeton University Press, 1985.

\bibitem{Ha}
A. Hatcher.
\newblock {\em  Notes on Basic 3-Manifolds Topology.}
\newblock http://www.math.cornell.edu/ hatcher/3M/3Mdownloads.html.





\bibitem{IchiharaMa}
Kazuhiro Ichihara and Jiming  Ma.
\newblock {\em A random link via bridge position is hyperbolic}.
 \newblock Topology Appl. 230 (2017), 131--138.


\bibitem{Ito}
Tetsuya  Ito.
\newblock {\em On a structure of random open books and closed braids}.
 \newblock Proc. Japan Acad. Ser. A Math. Sci. 91 (2015), no. 10, 160--162.

\bibitem{Ma2014}
 Jiming  Ma.
\newblock {\em The closure of a random braid is a hyperbolic link}.
 \newblock Proc. Amer. Math. Soc. 142 (2014), no. 2, 695--701.





\bibitem{Malyutin2018}
 A. V. Malyutin.
\newblock {\em  On the question of genericity of hyperbolic knots}.
 \newblock Int. Math. Res. Not. (2018). https://doi.org/10.1093/imrn/rny220.

\bibitem{Malyutin2019}
A. V. Malyutin.
\newblock {\em Hyperbolic links are not generic}.
\newblock  arxiv:1904.04458.

\bibitem{MS}
B. Mangum and T. Stanford.
\newblock {\em Brunnian links are determined by their complements.}
\newblock Algebraic \& Geometric Topology, Volume 1 (2001), 143--152.

\bibitem{M}
J. Milnor.
\newblock {\em  Link groups}.
\newblock Annals of Mathematics, 59 (1954), 177--195.

\bibitem{DR}
D. Rolfsen.
\newblock {\em Knots and Links}.
\newblock Mathematics lecture series, Volume 7, Publish or Perish, 1976.

\bibitem{Th}
W. Thurston.
\newblock {\em Three dimensional manifolds, Kleinian groups and hyperbolic geometry.}
\newblock   Bull. Amer. Math. Soc. 6 (1982), No. 3, 357--381.
\end{thebibliography}

\clearpage

\section{Appendix}
In this appendix, we prove Proposition \ref{prop:sumtree}, Proposition \ref{prop:fine}, and Lemma \ref{lem:isotopy}.

\begin{lem}\label{lem: hopf}
Let $L = L^1$$_{C_1} \dagger_{C_2} L^2$ be a Brunnian link and  $L^i- C_i$ is contained in the solid torus $V_i$ bounded by the decomposing torus $T$, for $i=1, 2$. If neither $L^1$ nor $L^2$ is Hopf link, then

    (i)Any meridian disk in $V_i$ intersects $L^i - C_i$ with at least 2 points, for $i=1, 2$;

    (ii) If a component of $L^1 - C_1$ bounds a disk that intersects $T$ with a meridian of $V_2$, then the disk intersects $T$ with at least two meridians of $V_2$.
\end{lem}

\begin{proof}
Note that if a link has two unknotted components, and one of which bounds a disk intersecting the other with exactly one point, then the link is Hopf link.

(i) Without loss of generality, suppose a meridian disk $D$ of $V_1$ intersects $L^1 - C_1$ with only
one point, say $p$, and $p$ is on a component $C \subset L^1$. Then $\partial D$ and $C$ form Hopf link. So by the Brunnian property, $L^1 - C_1 =C$. It follows that $T$ is $\partial$-parallel to $N(C)$.

(ii) Assume a component $C \subset L^1 - C_1$ bounds a disk $D$ such that only one of the intersection circles, say $\tilde{C} \subset D \cap T$, is a meridian of $V_2$. Then all the other circles are trivial on $T$. Eliminating all trivial circles by surgery from the innermost on $T$, we get a new disk $D'$ intersecting the core of $V_2$ with exactly one point. Hence this core and $C$ form Hopf link. Thus $L^1$ is Hopf link; this is impossible.
\end{proof}

\begin{proof}[Proof of Proposition \ref{prop:sumtree}]
The proof will be split into two parts, existence and uniqueness.

{\sc Existence}. Given a Brunnian link $L$, an s-sum decomposition with no Hopf link factor is equivalent to a collection of
disjoint unknotted essential tori $\{T_1, T_2, ... ,T_n\}$ in $S^3 -L$. Each $T_i$ decomposes $L$ as an s-sum. If some solid tori bounded by other tori are replaced by their cores, $T_i$ still decomposes this Brunnian link as an s-sum. There are no parallel tori in the collection since Hopf link is not s-prime. We add these tori in succession into $S^3$. Each step gives an s-sum decomposition.

Now in a step, suppose $L$ is decomposed as $L_1 (C_1)$ and $L_2 (C_2)$. If both $L_1$ and $L_2$ has more than two component, it's called a \emph{big step}. Otherwise, it's called a \emph{small step}. In each big step, $L_1$ and $L_2$ both have the numbers of components less than $L$. So there are finitely many big steps in total. In each small step, without loss of generality, we may assume $L_2$ has two components and $L_1$ has the same number of components as $L$.

Consider the disks bounded by a component of $L$ intersecting other components. We take the least number of the intersecting points among all such disks. The \emph{total intersection number} of $L$ is the sum of such numbers over all components of $L$. By Lemma \ref{lem: hopf}, the total intersection number of $L_1$ is less than that of $L$. The same holds when we decompose $L_2$. So there are finitely many small steps between two big steps. It follows that there are finitely many steps in total.

{\sc Uniqueness}. Suppose $\{T_1, T_2, ... ,T_n\}$ and $\{T'_1, T'_2, ... ,T'_l\}$ are two maximal collections of disjoint tori, both decomposing $L$ into s-prime Brunnian links. We consider each $T_i \cap T'_j$, a union of disjoint circles.

First consider circles inessential in both $T_i$ and $T'_j$. Choose an innermost circle $C_0$ on $T_i$, bounding $D_i$ and $D'_j$ on $T_i$ and $T'_j$ respectively. Then $D_i$ and $D'_j$ form a sphere. There is one \emph{empty} 3-ball bounded by this sphere, namely, the 3-ball does not intersect $L$, since $L$ is not split. So we can isotope $T'_j$ to push $D'_j$ across $D_i$ to eliminate $C_0$. We may assume there is
no circle inessential in both $T_i$ and $T'_j$.

Next consider circles inessential in one of $T_i$ and $T'_j$, say $T'_j$. Choose an innermost such
circle $C'$, bounding $D'$ on $T'_j$. Then $D'$ is the meridian disk of one solid torus bounded by
$T_i$, which is impossible by (i) $\Rightarrow$ (ii) in Proposition \ref{prop:sum}. So we may assume the
intersection circles are both essential on $T_i$ and $T'_j$, and thus are parallel.

Now the essential parallel circles cut both $T_i$ and $T'_j$ into annuli. Recall that any properly embedded annulus in a solid torus with essential boundary is $\partial$-parallel unless the boundaries are meridians. We may choose a solid torus $V$ bounded by $T_i$, such that the circles are not meridian of $V$. Consider an annulus $A' \subset T'_j$ with boundary on $T_i$, outermost in $V$. Then there is an annulus $A \subset T_i$ with the same boundary, and $A \cup A'$ bounds a compression solid torus $V_1$ for $A'$ in $V$. It follows that $V- V_1$ is also a solid torus.

If $V_1$ is empty, isotope $T'_j$ to delete $\partial A$. Otherwise, we denote the proper sublink in $V_1$ to be $L_1$, which is geometrically essential in $V_1$. We claim $V- V_1$ is empty. Assume otherwise. Then the proper sublink $L_2$ in $V - V_1$ is also geometrically essential. Notice that the sublink in $V$ and the core of $S^3 - V$ form a Brunnian link. Consequently when deleting $L_1$, $L_2$ bounds mutually disjoint disks, contradicting that $L_2$ is geometrically essential in $V - V_1$.

If $V_1$ is knotted, whose core is a torus knot, then $L_1$ is already nontrivial in $S^3$, a contradiction.
So $V_1$ is unknotted. This implies that the circles on $T_i$ are either $(1, n)$-curves or $(n, 1)$-curves.

{\sc Case~1}. They are $(1, n)$-curves. Then $T_i - A$ is also parallel to $A'$ in $V - V_1$. Since $V - V_1$ is empty, we can isotope $T_i$ across $A'$ to delete $\partial A'$.

{\sc Case~2}. They they are $(n, 1)$-curves. Then we can re-choose $V$ to be the other solid torus bounded by $T_i$ to transform into the first case.

Finally, we have that $T_1 \cup T_2 \cup ... \cup T_n$ and $T'_1 \cup T'_2 \cup ... \cup T'_l$ are disjoint. If there is a $T'_j$ not parallel to any $T_i$, we may add it into the first collection, contradicting the maximal choice. Thus they are parallel to each other up to orders and $n=l$.
\end{proof}

\begin{proof}[Proof of Proposition \ref{prop:fine}]
Suppose there are two collections of mutually disjoint fine companions $\mathcal{T}$ and $\mathcal{T'}$, where both have no parallel tori. By Haken's finiteness theorem, the number of complements in such collections is bounded.

The remainder of the proof is split into 2 steps. (1): After isotopy, $\mathcal{T}$ and $\mathcal{T'}$ are disjoint. (2): If the two collections are maximal, then they coincide after isotopy.

(1). Consider the intersection of $T_1 \in \mathcal{T}$ and $T'_1 \in \mathcal{T'}$ without loss of generality and denote the solid tori bounded by them to be $V$ and $V'$ respectively. Using the same argument as in the proof of the uniqueness in Proposition \ref{prop:sumtree}, we can easily eliminate intersection circles inessential on $T_1$ or $T'_1$.

Consider an annulus $A' \subset T'_1$ in $V$ with boundary on $T_1$, outermost in $V$. Then there is an annulus $A$ on $T_1$ with the same boundary.

{\sc Case~1}. $A'$ is $\partial$-parallel. Then we may assume $A \cup A'$ bounds a compression solid torus $V_1$ for $A'$ in $V$. If $V_1$ is empty, we delete $\partial A'$ as in the proof of Proposition \ref{prop:sumtree}. Suppose $V_1$ is not empty.

First we claim $V - V_1$ is empty, and thus $L \subset V_1$. In fact, otherwise, the proper sublink in $V - V_1$ is geometrically essential and thus nontrivial in $S^3$. Since $L$ is geometrically essential in $V$, it is geometrically essential in $V_1$. Now we discuss the circles $\partial A$ on $T_1$.

{\sc Subcase~1.1}. $\partial A$ are (0,1)-curves, namely, meridians of $V$. This is impossible since if  there is a compressing disk for $T'_1$ in $V$, then $T_1$ is not essential torus.

{\sc Subcase~1.2}. $\partial A$ are $(1, n)$-curves. Then $T_1 - A$ is also parallel to $A'$ in $V - V_1$. We can isotope $T_1$ across $A'$ to eliminate $\partial A'$.

{\sc Subcase~1.3}. Other cases. Then by surgery along $A$, we produce a new torus with $A'$ in $V$ which is not parallel to $T_1$, contradicting that $T_1$ is fine.

{\sc Case~2}. $A'$ is not $\partial$-parallel. Then $\partial A$ are $(0, 1)$-curves and $A' \cup (T_1 - A)$ bounds a knotted solid torus $\tilde{V}$ in $V$, with the same meridian. So the torus $A' \cup (T_1 - A)$ is
essential and thus $L \subset \tilde{V}$. Surgery along $T_1 - A$ produces a new torus with $A'$ in
$V$ which is not parallel to $T_1$, contradicting that $T_1$ is fine.

In summary, $T_1$ and $T'_1$ can be isotoped to be disjoint. Step by step we can make the fine companions in the two collections to be disjoint.

(2). Since the two collections are both maximal, they must be parallel to each other pair by pair.
\end{proof}

\begin{proof}[Proof of Lemma \ref{lem:isotopy}]
Consider the intersection circles of $T_1 \cup T_2 \cup ... \cup T_n$ and $T'_1 \cup
T'_2 \cup ... \cup T'_l$, the unions of tori in the two collections. Inessential circles can be deleted owing to that $L$ is not split and the sublinks in the corresponding solid tori are geometrically essential. Notice that that the proof of uniqueness in Proposition \ref{prop:sumtree} did not use that $T'_i$'s are unknotted tori. So the proof here is by the same token.
\end{proof}

\end{document}